\documentclass[12pt]{amsart}
\usepackage{amsmath}
\usepackage{amsfonts, amssymb, euscript, mathrsfs}
\usepackage{amsthm, upref}
\usepackage{graphicx}
\usepackage[usenames, dvipsnames]{color}

\usepackage{tikz}
\usepackage{enumerate}
\usepackage[tmargin=0.8in,bmargin=0.8in,lmargin=0.6in,rmargin=0.6in]{geometry}
\usepackage{aliascnt}
\usepackage{hyperref}
\usepackage{verbatim}

\allowdisplaybreaks[4]

\newcommand{\MZ}{\mathbb{Z}}
\newcommand{\vk}{\varkappa}

\newcommand{\BR}{\mathbb{R}}

\newcommand{\SL}{\sum\limits}

\newcommand{\be}{\beta}

\newcommand{\de}{\delta}

\newcommand{\La}{\Lambda}
\newcommand{\ME}{\mathbf E}

\newcommand{\CF}{\mathcal F}

\newcommand{\MP}{\mathbf P}

\newcommand{\CA}{\mathcal A}
\newcommand{\CL}{\mathcal L}
\newcommand{\CS}{\mathcal S}
\newcommand{\CD}{\mathcal D}

\newcommand{\CZ}{\mathcal Z}
\newcommand{\CX}{\mathcal X}

\newcommand{\Oa}{\Omega}

\newcommand{\si}{\sigma}

\newcommand{\pa}{\partial}
\renewcommand{\phi}{\varphi}

\newcommand{\la}{\lambda}
\newcommand{\fn}{\mathfrak{n}}
\newcommand{\Ra}{\Rightarrow}

\newcommand{\ol}{\overline}

\newcommand{\CM}{\mathcal M}

\newcommand{\norm}[1]{\lVert#1\rVert}
\renewcommand{\comment}[1]{}

\newcommand{\md}{\mathrm{d}}

\newcommand{\munit}{\mathbf{1}}
\newcommand{\mP}{\mathbf{p}}

\DeclareMathOperator{\SRBM}{SRBM}

\DeclareMathOperator{\diag}{diag}

\DeclareMathOperator{\tr}{tr}

\DeclareMathOperator{\TV}{TV}
\DeclareMathOperator{\mes}{mes}

\begin{document}

\theoremstyle{plain}
\newtheorem{thm}{Theorem}[section]
\newtheorem*{thmnonumber}{Theorem}
\newtheorem{lemma}[thm]{Lemma}
\newtheorem{prop}[thm]{Proposition}
\newtheorem{cor}[thm]{Corollary}
\newtheorem{open}[thm]{Open Problem}

\theoremstyle{definition}
\newtheorem{defn}{Definition}
\newtheorem{asmp}{Assumption}
\newtheorem{notn}{Notation}
\newtheorem{prb}{Problem}

\theoremstyle{remark}
\newtheorem{rmk}{Remark}
\newtheorem{exm}{Example}
\newtheorem{clm}{Claim}

\author{Andrey Sarantsev}

\title[Tail Estimates for the Stationary Distribution]{Reflected Brownian Motion in a Convex Polyhedral Cone:\\ Tail Estimates for the Stationary Distribution} 

\address{Department of Statistics and Applied Probability, University of California, Santa Barbara} 

\email{sarantsev@pstat.ucsb.edu}



\date{March 31, 2016. Version 10}

\begin{abstract}
Consider an multidimensional obliquely reflected Brownian motion in the positive orthant, or, more generally, in a convex polyhedral cone. We find sufficient conditions for existence of a stationary distribution and convergence to this distribution at the exponential rate, as time goes to infinity. We also prove that certain exponential moments for this distribution are finite, thus providing a tail estimate for this distribution. Finally, we apply these results to systems of rank-based competing Brownian particles.
\end{abstract}

\maketitle

\section{Introduction} A multidimensional obliquely reflected Brownian motion in a convex polyhedron $D \subseteq \BR^d$ has been extensively studied in the past few decades. This is a stochastic process $Z = (Z(t), t \ge 0)$ which takes values in $D$; in the interior of $D$, it behaves as a Brownian motion, and as it hits the boundary $\pa D$, 
it is reflected inside $D$, but not necessarily normally. For every {\it face} $D_i$ of the boundary $\pa D$, there is a vector $r_i$ on this face, pointing inside $r_i$, which governs the reflection. If $r_i$ is the inward unit normal vector to $D$, then this reflection is {\it normal}; otherwise, it is {\it oblique}. Special care should be taken for the reflection at the intersection of two or more faces. A formal definition is given in Sect. 2. 

One particularly important case is the {\it positive orthant} $D = \BR^d_+$, where $\BR_+ := [0, \infty)$. The concept of a semimartingale reflected Brownian motion (SRBM) in the orthant was introduced in \cite{HR1981a}, as a diffusion limit for series of queues, when traffic intensity at each queue tends to one (the so-called {\it heavy traffic limit}). Later, it was applied in the theory of competing Brownian particles (systems of rank-based Brownian particles, when each particle has drift and diffusion coefficients depending on its current ranking relative to the other particles), see \cite{BFK2005, Ichiba11}. The gap process (vector of gaps, or spacings, between adjacent particles) turns out to be an SRBM in the orthant. 

We refer the reader to the comprehensive survey \cite{Wil1995} about an SRBM in the orthant. Reflected Brownian motion in a convex polyhedron was introduced in \cite{DW1995} in a semimartingale form: {\it semimartingale reflected Brownian motion}, or an SRBM. The paper \cite{DW1995} contains a sufficient condition for weak existence and uniqueness in law; it is stated in Sect. 2. 



In this paper, we assume that the condition mentioned above holds; then an SRBM in a convex polyhedron $D$ exists and is unique in the weak sense, and versions of this SRBM starting from different points $x \in D$ form a Feller continuous strong Markov family. 

Of particular interest is a stationary distribution for an SRBM in a convex polyhedron $D$: a probability distribution $\pi$ on $D$ such that if $Z(0) \sim \pi$, then $Z(t) \sim \pi$ for all $t \ge 0$. This was a focus of extensive research throughout the last four decades. For the orthant (and, more generally, for a convex polyhedron), a necessary and sufficient condition for existence of a stationary distribution is not known. However, there are fairly general sufficient conditions and necessary conditions for the orthant, see \cite{DW1994, Chen1996}. For dimensions $d = 2$ and $d = 3$, a necessary and sufficient condition is actually known, see \cite{HR1993} and \cite[Appendix A]{HH2009} for $d = 2$, and \cite{BDH2010} for $d = 3$. For a convex polyhedron (more specifically, a convex polyhedral cone), see \cite{ABD2001, BD1999} for sufficient condition for existence of a stationary distribution. It was shown in \cite{DK2003} that if a stationary distribution exists, then it is unique. 

Exact form of this stationary distribution is known only in a few cases, the most important of which is the so-called {\it skew-symmetry condition}. Under this condition, the stationary distribution has a product-of-exponentials form, see \cite{HW1987a, Wil1995}. Other known cases (sums of products of exponentials) are studied in \cite{DM2009}. A necessary and sufficient condition for a probability distribution to be stationary is a certain integral equation, called the {\it Basic Adjoint Relationship}. However, it is not known how to solve this equation in the general case. 

We complement the results above by finding some new conditions for existence of a stationary distribution for an SRBM in the orthant and in a convex polyhedral cone. To this end, we find a {\it Lyapunov function}: this is a function $V : D \to [1, \infty)$ such that
for some constants $k, b > 0$ and a set $C \subseteq D$ which is ``small'' in a certain sense (we specify later what this means; for now, it is sufficient to take a compact $C$) the process
$$
V(Z(t)) - V(Z(0)) - \int_0^t\left(-kV(Z(s)) + b1_C(Z(s))\right)\md s
$$
is a supermartingale. This is a more general definition than is usually used (with the generator of an SRBM). Under some additional technical conditions (so-called irreducibility and aperiodicity, more on this later), if such function $V$ can be constructed, then there exists a unique stationary distribution $\pi$, and the SRBM $Z = (Z(t), t \ge 0)$ converges weakly to $\pi$ as $t \to \infty$; moreover, the convergence is exponentially fast in $t$. There is an extensive literature on Lyapunov functions and convergence, see \cite{MeynBook} for discrete-time Markov chains and \cite{MT1993a, MT1993b, DMT1995, DFG2009, FR2005} for continuous-time Markov processes. These methods were applied to an SRBM in the orthant in \cite{DW1994, BL2007} and to an SRBM in a convex polyhedral cone in \cite{ABD2001, BD1999}. However, in these articles they construct a Lyapunov function indirectly. In this article, we come up with an explicit formula:
$$
V(x) = e^{\la\phi(U(x))},\ \ U(x) := [x'Qx]^{1/2},\quad x \in D,
$$
where $Q$ is a $d\times d$ symmetric matrix such that $x'Qx > 0$ for $x \in D\setminus\{0\}$, $\la > 0$ is a certain constant (to be determined later), $\phi : \BR_+ \to \BR_+$ is a $C^{\infty}$ function with 
\begin{equation}
\label{eq:phi}
\phi(s) := \begin{cases}
0,\ s \le s_1;\\
s,\ s \ge s_2,
\end{cases}
\ \ \mbox{for some}\ \ 0 < s_1 < s_2.
\end{equation}
We can also conclude that $\int_DV(x)\pi(\mathrm{d}x) < \infty$. This explicit form of the function $V$ allows us to find tail estimates for the stationary distribution $\pi$. Let us also mention the paper \cite{HH2009}, which studies tail behavior of $\pi$ in case $d = 2$. 

The paper is organized as follows. In Sect. 2, we introduce all necessary concepts and definitions, explain the connection between Lyapunov functions, existence of a stationary distribution, and exponential convergence. In Sect. 3, we state the main result and compare it with already known conditions for existence of a stationary distribution; then, we prove this main result. Sect. 4 is devoted to systems of competing Brownian particles.

\subsection{Notation} We denote by $I_k$ the $k\times k$-identity matrix, and by $\munit$ the vector $(1, \ldots, 1)'$ (the dimension is clear from the context). For a vector $x = (x_1, \ldots, x_d)' \in \BR^d$, let $\norm{x} := \left(x_1^2 + \ldots + x_d^2\right)^{1/2}$ be its Euclidean norm. The norm of a $d\times d$-matrix $A$ is defined by 
$$
\norm{A} = \max\limits_{\norm{x} = 1}\norm{Ax} = \max\{\sqrt{\la}\mid \la \ \ \hbox{is an eigenvalue of}\ \ A'A\}.
$$
For any two vectors $x, y \in \BR^d$, their dot product is denoted by $x\cdot y = x_1y_1 + \cdots + x_dy_d$. As mentioned before, we compare vectors $x$ and $y$ componentwise: $x \le y$ if $x_i \le y_i$ for all $i = 1, \ldots, d$; $x < y$ if $x_i < y_i$ for all $i = 1, \ldots, d$; similarly for $x \ge y$ and $x > y$. This includes infinite-dimensional vectors from $\BR^{\infty}$. We compare matrices of the same size componentwise, too. For example, we write $x \ge 0$ for $x \in \BR^d$ if $x_i \ge 0$ for $i = 1, \ldots, d$; $C = (c_{ij})_{1 \le i, j \le d} \ge 0$ if $c_{ij} \ge 0$ for all $i$, $j$. 
Fix $d \ge 1$, and let $I \subseteq \{1, \ldots, d\}$ be a nonempty subset. Write its elements in the order of increase: $I = \{i_1, \ldots, i_m\},\ \ 1 \le i_1 < i_2 < \ldots < i_m \le d$. For any $x \in \BR^d$, let $[x]_I := (x_{i_1}, \ldots, x_{i_m})'$. For any $d\times d$-matrix $C = (c_{ij})_{1 \le i, j \le d}$, let $[C]_I := \left(c_{i_ki_l}\right)_{1 \le k, l \le m}$. A one-dimensional Brownian motion with zero drift and unit diffusion, starting from $0$, is called a {\it standard Brownian motion}. The symbol $\mes$ denotes the Lebesgue measure on $\BR^d$. We write $f \in C^{\infty}(D)$ for an infinitely differentiable function $f : D \to \BR$. Take a measurable space $(\mathfrak X, \nu)$. For any measurable function $f : \mathfrak X \to \BR$, we denote $(\nu, f) := \int_{\mathfrak X}f\md\nu$. For a measurable function $f : \mathfrak X \to [1, \infty)$, define the norm
$\norm{\nu}_f := \sup\left|(\nu, g)\right|$, where the supremum is taken over all measurable functions $g : \mathfrak X \to \BR$ such that $|g(x)| \le f(x)$ for all $x \in \mathfrak X$. For $f = 1$, this is the {\it total variation norm}: $\norm{\nu}_{\TV}$.

\section{Definitions and Background}

\subsection{Definition of an SRBM in a Convex Polyhedron} 

Fix the dimension $d \ge 2$, and the number $m$ of edges. Take $m$ unit vectors $\fn_1, \ldots, \fn_m$, and $m$ real numbers $b_1, \ldots, b_m$. Consider the following domain:
$$
D := \{x \in \BR^d\mid x\cdot\fn_i \ge b_i,\ i  =1,\ldots, m\}.
$$
We assume that each {\it face} $D_i$ of the boundary $\pa D$:
$$
D_i := \{x \in \BR^d\mid x\cdot\fn_i = b_i,\ x\cdot\fn_j \ge b_j,\ j = 1, \ldots, m,\ j \ne i\},\quad i = 1,\ldots, m,
$$
is $(d-1)$-dimensional, and the interior of $D$ is nonempty. Then $D$ is called a {\it convex polyhedron}. For each face $D_i$, $\fn_i$ is the inward unit normal vector to this face. Define the following $d\times m$-matrix: $N = [\fn_1|\ldots|\fn_m]$. Now, take a vector $\mu \in \BR^d$ and a positive definite symmetric $d\times d$-matrix $A$. Consider also an $m\times d$-matrix $R = [r_1|\ldots|r_m]$, with $\fn_i\cdot r_i = 1$ for $i = 1, \ldots, m$. We are going to define a process $Z = (Z(t), t \ge 0)$ with values in $D$, which behaves as a $d$-dimensional Brownian motion with drift vector $\mu$ and covariance matrix $A$, so long as it is inside $D$; at each face $D_i$, it is reflected according to the vector $r_i$. First, we define its deterministic version: a {\it solution to the Skorohod problem}. 

\begin{defn} Take a continuous function $\CX : \BR_+ \to \BR^d$ with $\CX(0) \in D$. A {\it solution to the Skorohod problem in $D$} with {\it reflection matrix} $R$ and {\it driving function} $\CX$ is any continuous function $\CZ : \BR_+ \to D$ such that there exist $m$ continuous functions $\CL_1, \ldots, \CL_m : \BR_+ \to \BR_+$ which satisfy the following properties:

(i) each $\CL_i$ is nondecreasing, $\CL_i(0) = 0$, and can increase only when $\CZ \in D_i$; we can write the latter property as
$$
\CL_i(t) = \int_0^t1(Z(s) \in D_i)\md\CL_i(s),\quad t \ge 0;
$$

(ii) for all $t \ge 0$, we have:
$$
\CZ(t) = \CX(t) + R\CL(t), \hbox{where} \CL(t) = (\CL_1(t), \ldots, \CL_m(t))'.
$$

The function $\CL_i$ is called the {\it boundary term} corresponding to the face $D_i$. 
\end{defn}

For the rest of the article, assume the usual setting: a filtered probability space $(\Oa, \CF, (\CF_t)_{t \ge 0}, \MP)$ with the filtration satisfying the usual conditions. 

\begin{defn} Fix $z \in D$. Take an $((\CF_t)_{t \ge 0}, \MP)$-Brownian motion $W = (W(t), t \ge 0)$ with drift vector $\mu$ and covariance matrix $A$, starting from $z$. 
A continuous adapted $D$-valued process $Z = (Z(t), t \ge 0)$, which is a solution to the Skorohod problem in $D$ with reflection matrix $R$ and driving function $W$, is called a {\it semimartingale reflected Brownian motion} (SRBM) {\it in $D$}, with {\it drift vector} $\mu$, {\it covariance matrix} $A$, and {\it reflection matrix} $R$, {\it starting from} $z$. It is denoted by $\SRBM^d(D, R, \mu, A)$. For the case $D = \BR^d_+$, we denote it simply by $\SRBM^d(R, \mu, A)$. 
\label{defn:SRBM}
\end{defn}

We shall present a sufficient condition for existence and uniqueness taken from \cite{DW1995}. First, let us introduce a concept concerning the geometry of the polyhedron $D$. 

\begin{defn} For a nonempty subset $I \subseteq \{1,\ldots, m\}$, let $D_I := \cap_{i \in I}D_i$, and let $D_{\varnothing} := D$. A nonempty subset $I \subseteq \{1, \ldots, m\}$ is called {\it maximal} if $D_I \ne \varnothing$ and for $I \subsetneq J \subseteq \{1,\ldots, m\}$ we have: $D_J \subsetneq D_I$. 
\end{defn}

Now, let us define certain useful classes of matrices.

\begin{defn} Take a $d\times d$-matrix $M = (m_{ij})_{1 \le i, j \le d}$. It is called an {\it $\CS$-matrix} if for some $u \in \BR^d$, $u > 0$ we have: $Mu > 0$. It is called {\it completely-$\CS$} if for every nonempty $I \subseteq \{1,\ldots, d\}$ we have: $[M]_I$ is an $\CS$-matrix. It is called a $\CZ$-matrix if $m_{ij} \le 0$ for $i \ne j$. It is called a {\it reflection nonsingular $\CM$-matrix} if it is both a completely-$\CS$ and a $\CZ$-matrix with diagonal elements equal to one: $r_{ii} = 1$, $i = 1, \ldots, d$. It is called {\it strictly copositive} if $x'Mx > 0$ for $x \in \BR^d_+\setminus\{0\}$. It is called {\it nonnegative} if all its elements are nonnegative. 
\end{defn}

A useful equivalent characterization of reflection nonsingular $\CM$-matrices is given in \cite[Lemma 2.3]{MyOwn3}. Now, let us finally state the existence and uniqueness result, taken from \cite{DW1995}. 

\begin{prop} Assume that for every maximal set $I \subseteq \{1, \ldots, m\}$ the matrices $[N'R]_I$ and $[R'N]_I$ are $\CS$-matrices. Then for every $z \in D$, there exists a weak version of an $\SRBM^d(D, R, \mu, A)$, and it is unique in law. Moreover, these processes for $z \in D$ form a Feller continuous strong Markov family. 
\label{prop:weak-existence}
\end{prop}

\begin{rmk} For a particular important case of the positive orthant: $D = \BR_+^d$, that is, when $m = d$, $\fn_i = e_i$ and $b_i = 0$ for $i = 1, \ldots, d$, we have: $N = I_d$, every nonempty subset $I \subseteq \{1, \ldots, d\}$ is maximal, and the condition from Proposition~\ref{prop:weak-existence} is equivalent to the matrix $R$ being completely ${\mathcal {S}}$ (because $R$ is completely ${\mathcal {S}}$ if and only if $R'$ is completely ${\mathcal {S}}$). This turns out to be not just sufficient but a necessary condition, see \cite{Wil1995}. 
\end{rmk} 

A sufficient condition for strong existence and pathwise uniqueness was found in \cite{HR1981a} for the orthant: $R$ must be a reflection nonsingular-$\CM$ matrix. 
 However, we shall not need strong existence and pathwise uniqueness in this paper. The generator of this process is given by
$$
\CA f(x) := \mu\cdot\nabla f(x) + \frac12\SL_{i=1}^d\SL_{j=1}^da_{ij}\frac{\pa^2f(x)}{\pa x_i\pa x_j},
$$
with the domain $\CD(\CA)$ containing the following subset of functions:
$$
\CD(\CA) \supseteq \{f \in C^{\infty}(D)\mid \left.r_i\cdot\nabla f(x)\right|_{x \in D_i} = 0,\ \ i = 1, \ldots, m\}.
$$

\subsection{Recurrence of Continuous-Time Markov Processes} Let us remind the basic concepts of recurrence, irreducibility and aperiodicity for continuous-time Markov processes. This exposition is taken from \cite{MT1993a, MT1993b, DFG2009, DMT1995, BL2007}. Take a locally compact separable metric space $\mathfrak X$ and denote by $\mathfrak B$ its Borel $\si$-field. 
Let 
$$
\left(\Oa, \CF, (\CF_t)_{t \ge 0}, (X(t), t \ge 0), (\MP_x)_{x \in \mathfrak X}\right)
$$
be a time-homogeneous Markov family, where $X(t)$ has continuous paths under each measure $\MP_x$. Denote by $P^t(x, A) = \MP_x(X(t) \in A)$ the transition function, and by $\ME_x$ the expectation operator corresponding to $\MP_x$. Denote by $P^tf$ and $\nu P^t$ the action of this transition semigroup on functions $f : \mathfrak X \to \BR$ and Borel measures on $\mathfrak X$. Take a $\si$-finite {\it reference measure} $\nu$ on $\mathfrak X$. The process $X$ is called {\it $\nu$-irreducible} if for $A \in \mathfrak B$ we have:
$$
\nu(A) > 0\ \ \Ra \ \ \ME_x\left[\int_0^{\infty}1_A(X(s))\md s\right] > 0\quad \hbox{for all}\ \ x \in \mathfrak X.
$$
If such measure exists, then there is a maximal irreducibility measure $\psi$ (every other irreducibility measure is absolutely continuous with respect to $\psi$), which is unique up to equivalence of measures. A set $A \in \mathfrak B$ with $\psi(A) > 0$ is {\it accessible}. A nonempty $C \in \mathfrak B$ is {\it petite} if there exists a probability distribution $a$ on $\BR_+$ and a nontrivial $\si$-finite measure $\nu_a$ on $\mathfrak B$ such that
$$
\forall x \in C,\ \ \int_0^{\infty}P^t(x, \cdot)a(\md t) \ge \nu_a(\cdot).
$$
Suppose that, in addition, this distribution $a$ is concentrated at one point $t > 0$: $a = \de_t$. Equivalently, there exists a $t > 0$ and a  nontrivial $\si$-finite measure $\nu_a$ on $\mathfrak B$ such that
$$
\forall x \in C,\ \ \int_0^{\infty}P^t(x, \cdot)a(\md t) \ge \nu_a(\cdot).
$$
Then the set $C$ is called {\it small}. The process is {\it Harris recurrent} if, for some $\si$-finite measure $\nu$, 
$$
\nu(A) > 0\ \ \Ra\ \ \int_0^{\infty}1_A(X(s))\md s = \infty\ \ {\mathbf {P}}_x-\mbox{a.s.}\quad \hbox{for all}\ \ x \in \mathfrak X.
$$
Harris recurrence implies $\nu$-irreducibility. A Harris recurrent process possesses an invariant measure $\pi$, which is unique up to multiplication by a constant. If $\pi$ is finite, then it can be scaled to be a probability measure, and in this case the process is called {\it positive Harris recurrent}. An irreducible process is {\it aperiodic} if there exists an accessible petite set $C$ and $T > 0$ such that for all $x \in C$ and $t \ge T$, we have: $P^t(x, C) > 0$.

\begin{defn} The process $X$ is called {\it $V$-uniformly ergodic} for a function $V : \mathfrak X \to [1, \infty)$ if it has a unique stationary distribution $\pi$, and there exists constants $K, \vk > 0$ such that for all $x \in \mathfrak X$ and $t \ge 0$ we have:
$$
\norm{P^t(x, \cdot) - \pi(\cdot)} \le KV(x)e^{-\vk t}.
$$
\end{defn}

Now, let us state a few auxillary statements. The next proposition was proved in \cite[Chapter 6]{MeynBook} for discrete-time processes, but the proof is readily transferred to continuous-time setting. 

\begin{prop} For a Feller continuous strong Markov family, every compact set is petite.
\label{prop:cpt-petite}
\end{prop}

\begin{lemma} Take a Feller continuous strong Markov family. Assume $\psi$ is a reference measure such that there exists a compact set $C$ with $\psi(C) > 0$. If $P^t(x, A) > 0$ for all $t > 0$, $x \in \mathfrak X$ and $A \in \mathfrak B$ such that $\psi(A) > 0$, then the process is $\psi$-irreducible and aperiodic.
\label{lemma:cond-for-irr-aper}
\end{lemma}

\begin{proof} Irreducibility follows from the definition. For aperiodicity, we can take the compact set $C$, because it is petite by Proposition~\ref{prop:cpt-petite}. If $\psi'$ is a maximal irreducibility measure, then $\psi(C) > 0$ and $\psi \ll \psi'$, and so $\psi'(C) > 0$. The rest is trivial.
\end{proof} 

Finally, the following statement was proved in \cite[Lemma 3.4]{DK2003}.

\begin{prop} For an $\SRBM^d(D, R, \mu, A)$ under the conditions of Proposition~\ref{prop:weak-existence}, for every $t > 0$, $x \in D$, and $A \subseteq D$ with $\mes(A) > 0$ we have: $P^t(x, A) > 0$. 
\label{prop:srbm-irr-aper}
\end{prop}

\begin{rmk} Combining Lemma~\ref{lemma:cond-for-irr-aper} and Proposition~\ref{prop:srbm-irr-aper}, we get that an $\SRBM^d(D, R, \mu, A)$ is irreducible and aperiodic.
\label{rmk:final}
\end{rmk}

\subsection{Lyapunov Functions and Exponential Convergence} There is a vast literature (some of these were mentioned in Sect. 1) on connection between Lyapunov functions for Markov processes and their convergence to the stationary distribution. 
However, for the purposes of this article, we need to state the result is a slightly different form. First, let us define the concept of a Lyapunov function.

\begin{defn} Take a continuous function $V : \mathfrak X \to [1, \infty)$. Suppose there exists a closed petite set $C \subseteq \mathfrak X$ and constants $k, b > 0$ such that the process
\begin{equation}
\label{eq:supermart}
V(X(t)) - V(X(0)) - \int_0^t\left[-kV(X(s)) + b1_{C}(X(s))\right]\md s
\end{equation}
is an $((\CF_t)_{t \ge 0}, \MP_x)$-supermartingale for all $x \in \mathfrak X$. 
If, in addition, $\sup_CV < \infty$, then $V$ is called a {\it Lyapunov function} for the process $X$. 
\label{defn:Lyapunov-func}
\end{defn}

\begin{rmk} Equivalently, we can request that the process in~\eqref{eq:supermart} is a local supermartingale. This is equivalent to it being a supermartingale, because this process is bounded from below by $-V(x) - bT$ on any time interval $[0, T]$ under the measure $\MP_x$. (Every local supermartingale which is bounded from below is a true supermartingale; this follows from a trivial application of Fatou's lemma.) 
\end{rmk} 

This definition is taken from \cite[Section 3]{DFG2009} with minor adjustments, with 
$\phi(s) = ks$ in the notation of \cite{DFG2009}. This is a slightly more general definition than is often stated in the literature; a more customary one invloves the generator of the Markov family. First, let us state an auxillary lemma.

\begin{lemma} For some constant $c_6 > 0$, we have: 
$P^sU(x) \le c_6U(x)$, for all $x \in \mathfrak X$ and $s \in [0, 1]$. 
\label{lemma:estimate}
\end{lemma} 

\begin{proof} Because $U$ and $V$ are equivalent in the sense of~\eqref{eq:equivalent-functions}, it suffices to prove the statement of Lemma~\ref{lemma:estimate} for $V$ instead of $U$. But this follows from the fact that the process~\eqref{eq:supermart} is a supermartingale. Indeed, take $\ME_x$ in~\eqref{eq:supermart} and get:
$$
P^tV(x) - V(x) + k\int_0^tP^sV(x) - b\int_0^t\MP^s(x, C)\md s \le 0.
$$
Therefore, $P^tV(x) \le V(x) + bt$. But $V(x) \ge 1$, so for $t \in [0, 1]$, we get:
$P^tV(x) \le (1 + b)V(x)$. This completes the proof of Lemma~\ref{lemma:estimate}. 
\end{proof}

Next, we present the main result for this subsection.

\begin{thm} Assume there exists a Lyapunov function $V$, and the process is irreducible and aperiodic. Then there exists a unique stationary distribution $\pi$, the process is $V$-uniformly ergodic, and we have the following estimate:
\begin{equation}
\label{eq:general-tail-estimate}
(\pi, V) \equiv \int_{\mathfrak X}V(x)\pi(\mathrm{d}x) < \infty.
\end{equation}
\end{thm}

\begin{proof} Existence and uniqueness of $\pi$ together with~\eqref{eq:general-tail-estimate} follows from \cite[Proposition 3.1]{DFG2009}. If the process is irreducible, then the skeleton chain $(X(n))_{n \in \MZ_+}$ is irreducible. Apply \cite[Theorem 3.3]{DFG2009} to the case $\phi(x) = kx$, we get that for any $t_0 > 0$, there exists a function $\tilde{V} : \CX \to [k, \infty)$, an accessible petite set $\tilde{C}$ for the skeleton chain $(X(n))_{n \in \MZ_+}$ and a constant $\tilde{b} > 0$ such that $\sup_{\tilde{C}}\tilde{V} < \infty$, 
$$
0 < c_1 \le \frac{\tilde{V}(x)}{V(x)} \le c_2 < \infty,\ \ x \in \mathfrak X,
\ \ \mbox{and}\ \ 
P^{1}\tilde{V} \le (1 - k)\tilde{V}  + \tilde{b}1_{\tilde{C}}.
$$
Taking $U := \tilde{V}/k : \mathfrak X \to [1, \infty)$, we get: there exists $\la := 1 - k < 1$ and $b' = \tilde{b}/k > 0$ such that $P^{T}U \le -\la U + b'1_{\tilde{C}}$, and
\begin{equation}
\label{eq:equivalent-functions}
0 < c_3 \le \frac{U(x)}{V(x)} \le c_4 < \infty,\ \ x \in \mathfrak X.
\end{equation}
It follows from Proposition~\ref{prop:srbm-irr-aper} that the skeleton chain $(X(n))_{n \ge 0}$ is irreducible and aperiodic. By \cite[Theorem 5.5.7]{MeynBook}, the petite set $\tilde{C}$ is small for this skeleton chain. Since this chain is irreducible and aperiodic, by \cite[Theorem 2.1(c)]{DMT1995} for some constants $c_5 > 0$ and $\rho \in (0, 1)$, we have: 
\begin{equation}
\label{eq:exp-ergodicity-for-skeleton}
\norm{P^{n}(x, \cdot) - \pi(\cdot)}_{U} \le c_5U(x)\rho^n.
\end{equation}
Next, we follow the proof of \cite[Theorem 5.2]{DMT1995}. Every $t \ge 0$ can be represented as $t = n + s$, where $n \in \MZ_+$, $s \in [0, 1)$. Since $\pi$ is stationary, we have: $\pi P^s = \pi$. Therefore, for any measurable $g : \mathfrak X \to [1, \infty)$ with $|g(x)| \le U(x)$, 
$$
P^tg(x)  - (\pi, g) = P^nP^sg(x) - (\pi, P^sg).
$$
But from Lemma~\ref{lemma:estimate} we have:
$$
\left|P^sg(z)\right| \le P^sU(z) \le c_6U(z),\ \ z \in \mathfrak X.
$$
From~\eqref{eq:exp-ergodicity-for-skeleton}, because $|g(x)| \le U(x)$ for $x \in \mathfrak X$, we get:
$$
\left|P^nP^sg(x) - (\pi, P^sg)\right| \le c_5c_6U(x)\rho^n.
$$
Since $n \le t - 1$,  
$$
\left|P^tg(x)  - (\pi, g)\right| \le c_5c_6\rho^{-1}U(x)\rho^t.
$$
This proves that, for $c_7 := c_5c_6\rho^{-1}$,  
$$
\norm{P^t(x, \cdot) - \pi(\cdot)}_{U} \le c_7U(x)e^{-\vk t},\ \ \vk := -\ln\rho.
$$
This is $U$-uniform ergodicity. Since the functions $U$ and $V$ are equivalent in the sense of~\eqref{eq:equivalent-functions}, this also means $V$-uniform ergodicity. 
\end{proof}

\section{Main Results} 

\subsection{Statement of the General Result} Consider now a special type of a convex polyhedron, namely a {\it convex polyhedral cone}: 
$D = \{x \in \BR^d\mid Nx \ge 0\}$, where $N$ is a $m\times d$-matrix, constructed in Sect. 2.2. This fits into the general framework of Definition~\ref{defn:SRBM}, if we let $b_1 = \ldots = b_m = 0$. 
What follows is the main result of the paper. 

\begin{thm}
\label{thm:main}
Suppose that conditions of Proposition~\ref{prop:weak-existence} hold. 
Assume there exists a symmetric nonsingular $d\times d$-matrix $Q$ such that:

(i) $x'Qx > 0$ for $x \in D\setminus\{0\}$;

(ii) $(R'Qx)_j \le 0$ for $x \in D_j$, for each $j = 1,\ldots, m$;

(iii) $x'Q\mu < 0$ for $x \in D\setminus\{0\}$. 

Take a $C^{\infty}$ function $\phi : \BR_+ \to \BR_+$ defined in~\eqref{eq:phi}. 
Denote 
\begin{equation}
\label{eq:La}
\La := 2\min\limits_{x \in D\setminus\{0\}}\frac{|Q\mu\cdot x|U(x)}{x'QAQx}.
\end{equation}
Then for $\la \in (0, \La)$, the function
\begin{equation}
\label{eq:V-la}
V_{\la}(x) = e^{\la\phi(U(x))},\ \ U(x) := \left[x'Qx\right]^{1/2},
\end{equation}
is a Lyapunov function for the $\SRBM^d(D, R, \mu, A)$. Therefore, the $\SRBM^d(D, R, \mu, A)$ has a unique stationary distribution $\pi$, which satisfies 
\begin{equation}
\label{eq:tail-estimate-SRBM}
(\pi, V_{\la}) \equiv \int_DV_{\la}(x)\pi(\mathrm{d}x) < \infty,
\end{equation}
and is $V_{\la}$-uniformly ergodic. 
\end{thm}

\begin{rmk} The quantity $\La$ is strictly positive. Indeed, the matrix $A$ is positive definite, and $Q$ is nonsingular; so for $x \ne 0$ we have: $Qx \ne 0$ and $x'QAQx = (Qx)'A(Qx) > 0$. Also, $Q\mu\cdot x = x'Q\mu < 0$, and $U(x) > 0$ for $x \in D\setminus\{0\}$. Therefore, the fraction is positive for each $x \in D\setminus\{0\}$. Since this fraction is homogeneous (invariant under scaling), we can take the minimum on the compact set $\{x \in D\mid \norm{x} = 1\}$. The rest is trivial. 
\end{rmk}

The estimate~\eqref{eq:tail-estimate-SRBM} implies that some exponential moments of $\pi$ are finite. Namely, let 
\begin{equation}
\label{eq:K}
K := \min\limits_{\substack{x \in D\\ \norm{x} = 1}}U(x).
\end{equation}
This quantity is strictly positive, because $U(x) > 0$ on the compact set $\{x \in D\mid \norm{x}  =1\}$. Therefore, for large enough $\norm{x}$ we have:
$$
V_{\la}(x) \ge e^{\la K\norm{x}},
$$
and 
$$
\int_De^{\rho\norm{x}}\pi(\mathrm{x}) < \infty\ \ \mathrm{for}\ \ \rho \in (0, \La K).
$$
From here, we get: for every $a \ge 0$, 
$$
\pi\{x \in D\mid \norm{x} \ge a\} \le C(\rho)e^{-a\rho}\ \ \mathrm{for}\ \ \rho \in (0, \La K).
$$

Let us compare this result with \cite{ABD2001, BD1999}, where a more general case is considered (drift vector and covariance matrix depend on the state). There, a sufficient condition for $V$-uniform ergodicity is:

\medskip

(i) that Skorohod problem in $D$ has a unique solution for every driving function and is Lipschitz continuously dependent on this function, in the metric of $C([0, T], \BR^d)$ for every $T > 0$;

\medskip

(ii) there exists a vector $b \in \BR^m, b > 0$, such that $Rb = -\mu$. 

\medskip

Condition (i) is stronger than the one from Proposition~\ref{prop:weak-existence}. However, we were not able to come up with an example when conditions of Theorem~\ref{thm:main} hold, but condition (i) does not hold. Some sufficient conditions for (i) to hold are known from \cite{DI1991}. However, condition (ii) is much simpler than (i) - (iii) from Theorem~\ref{thm:main}. The results from \cite{ABD2001, BD1999} also construct a Lyapunov function indirectly, without giving an explicit formula. This does not allow to construct explicit  tails estimates, as in \ref{thm:main}. 

\subsection{Applications to the Case of the Positive Orthant} Now, let $D = \BR^d_+$, that is, $m = d$ and $N = I_d$. We have the following immediate corollary of Theorem~\ref{thm:main}.

\begin{cor}
\label{cor:main} Assume $R$ is a completely ${\mathcal {S}}$ matrix. Suppose there exists a strictly copositive nonsingular $d\times d$-matrix $Q$ such that $QR$ is a $\CZ$-matrix, and $Q\mu < 0$. Then an $\SRBM^d(R, \mu, A)$ has a unique stationary distribution $\pi$ and is $V_{\la}$-uniformly ergodic for $\la \in (0, \La)$, while $\pi$ satisfies~\eqref{eq:tail-estimate-SRBM}. Here, $V_{\la}$ is defined in~\eqref{eq:V-la}, and $\La$ is defined in~\eqref{eq:La}. 
\end{cor}

\begin{proof} Condition (i) of Theorem~\ref{thm:main} follows from the definition of copositivity. Condition (ii) follows from the assumption that $QR$ is a $\CZ$-matrix, because then for $z \in D_i$ we have: $z \ge 0$, but $z_i = 0$, and so 
$$
(QRz)_i = \SL_{j=1}^d(QR)_{ij}z_j = \SL_{j \ne i}(QR)_{ij}z_j \le 0.
$$
Condition (iii) follows from $Q\mu < 0$. 
\end{proof}

A particular example of this is as follows.

\begin{cor}
\label{cor:nonsingular-M}
Assume $R$ is a $d\times d$-reflection nonsingular $\CM$-matrix, and there exists a diagonal matrix $C = \diag(c_1, \ldots, c_d)$ with $c_1, \ldots, c_d > 0$ such that $\ol{R} = RC$ is symmetric. If $R^{-1}\mu < 0$, then the process 
$\SRBM^d(R, \mu, A)$ has a unique stationary distribution $\pi$, and is $V_{\la}$-uniformly ergodic with 
$$
V_{\la}(x) = e^{\la\phi(U(x))},\ \ U(x) := \left[x'\ol{R}^{-1}x\right]^{1/2}
$$
for $\la \in (0, \La)$, where the function $\phi$ is defined in~\eqref{eq:phi}, and
$$
\La := 2\min\limits_{x \in \BR^d_+\setminus\{0\}}\frac{|\ol{R}^{-1}\mu\cdot x|U(x)}{x'\ol{R}^{-1}A\ol{R}^{-1}x}.
$$
In addition, $(\pi, V_{\la}) < \infty$ for $\la \in (0, \La)$. 
\end{cor}

\begin{proof} Just take $Q = \ol{R}^{-1} = C^{-1}R^{-1}$ in Corollary~\ref{cor:main}. Let us show that the matrix $Q$ is strictly copositive. From \cite[Lemma 2.3]{MyOwn3}, $R^{-1}$ is a nonnegative matrix with strictly positive elements on the main diagonal. Since $C^{-1}$ is a diagonal matrix with strictly positive elements on the diagonal, the matrix $\ol{R}^{-1}$ is also a nonnegative matrix with strictly positive elements on the main diagonal. Therefore, for $x \in \BR^d_+$, $x \ne 0$ we have:
$x'\ol{R}^{-1}x > 0$. Now, from $R^{-1}\mu < 0$ it follows that $\ol{R}^{-1}\mu < 0$. The rest is trivial. 
\end{proof}

\begin{exm} However, Corollary~\ref{cor:main} can be applied not only to the case when $R$ is a reflection nonsingular $\CM$-matrix. Indeed, let $d = 2$ and
$$
R = 
\begin{bmatrix}
1 & 0.5\\
0.5 & 1
\end{bmatrix}
\ \ \
\mu = \begin{bmatrix}-1\\ -1\end{bmatrix}
$$
Then $R$ is a completely ${\mathcal {S}}$ matrix. Take the matrix
$$
Q = 
\begin{bmatrix}
1 & -0.6\\
-0.6 & 1
\end{bmatrix}
\ \ \mathrm{then}\ \ 
QR = \begin{bmatrix} 0.7 & -0.1\\ -0.1 & 0.7\end{bmatrix}
$$
is a $\CZ$-matrix, and $Q\mu < 0$. However, $R$ is not a reflection nonsingular $\CM$-matrix. 
\end{exm}

It is instructive to compare these results with already known ones. It turns out that the only new statement in Corollary~\ref{cor:nonsingular-M} is the tail estimate $(\pi, V_{\la}) < \infty$ for an explicitly constructed function $V_{\la}$. Existence (and uniqueness) of a stationary distribution and $V$-uniform ergodicity for some function $V : \BR^d_+ \to [1, \infty)$ are already known from \cite{DW1994}, \cite{BL2007} (only there is no simple formula for the Lyapunov function $V$ in these papers: it is simply known that $V(x) \ge a_1e^{a_2\norm{x}}$ for some $a_1, a_2 > 0$). The paper \cite{DW1994} states the {\it fluid path condition}, which is sufficient for $V$-uniform ergodicitiy: for every $x \in \BR^d_+$, any solution of the Skorohod problem in the orthant with reflection matrix $R$ and driving function $x + \mu t$ must tend to zero as $t \to \infty$. This turns out to be a necessary and sufficient condition for the case $d =3$. In the case $d = 2$, another necessary and sufficient condition is found: $R$ must be nonsingular and $R^{-1}\mu < 0$, see \cite{HR1993} and \cite[Appendix A]{HH2009}. In fact, the following condition is {\it necessary} for existence of a stationary distribution: $R$ is nonsingular and $R^{-1}\mu < 0$, see \cite[Appendix C]{BDH2010}. For $d = 3$, the fluid path condition is weaker than this necessary condition, see \cite{BDH2010}.

It is not known for $d \ge 4$ whether the fluid path condition is necessary. Therefore, Corollary~\ref{cor:main} might contain results which are new compared to the fluid path condition. However, we do not know any counterexamples to fluid path condition (that is, cases when it is false, but the stationary distribution exists). This is a matter for future research.

\subsection{Proof of Theorem~\ref{thm:main}} Recall from Remark~\ref{rmk:final} that an $\SRBM^d(D, R, \mu, A)$ is irreducible and aperiodic. The rest of the proof will be devoted to proving that the function~\eqref{eq:V-la} is indeed a Lyapunov function in the sense of Definition~\ref{defn:Lyapunov-func}. Apply the It\^o-Tanaka formula to 
$$
Z(t) = W(t) + RL(t),\ \ t \ge 0,
$$
where $Z = (Z(t), t \ge 0)$ is an $\SRBM^d(D, R, \mu, A)$, $W = (W(t), t \ge 0)$ is the driving Brownian motion for $Z$, and $L = (L_1, \ldots, L_m)'$ is the vector of boundary terms. Because of~\eqref{eq:phi}, for $x \in D$ such that $\norm{x} \ge s_2$, we have:
$$
V_{\la}(x) = e^{\la U(x)}.
$$
First, let us calculate the first- and second-order partial derivatives of $U$ on this set.
Since $Q$ is symmetric and $x'Qx > 0$ for $x \in D$ such that $\norm{x} \ge s_2$, we have:
$$
\frac{\pa(x'Qx)}{\pa x_i} = 2(Qx)_i,\quad i = 1, \ldots, d.
$$
Therefore, 
$$
\frac{\pa U(x)}{\pa x_i} = \frac1{2U(x)}\frac{\pa U(x)}{\pa x_i} = \frac{(Qx)_i}{U(x)}, \quad i = 1, \ldots, d.
$$
Now, 
\begin{align*}
\frac{\pa^2 U(x)}{\pa x_i\pa x_j} =& \frac{\frac{\pa(Qx)_i}{\pa x_j}U(x) - (Qx)_i\frac{\pa U(x)}{\pa x_j}}{U^2(x)} = \frac{q_{ij}U(x) - (Qx)_i\frac{(Qx)_j}{U(x)}}{U^2(x)} \\ & = \frac{q_{ij}U^2(x) - (Qx)_i(Qx)_j}{U^3(x)} = \frac{1}{U^3(x)}\left(q_{ij}(x'Qx) - (Qxx'Q)_{ij}\right).
\end{align*}
As $\norm{x} \to \infty$, these second-order derivatives tend to zero, because $U(x) \ge K\norm{x}$ for $x \in D$. Now, let us calculate the first- and second-order partial derivatives for $V_{\la}$:
$$
\frac{\pa V_{\la}(x)}{\pa x_i} = \la\frac{\pa U}{\pa x_i}V_{\la}(x) = \la V_{\la}(x)\frac{(Qx)_i}{U(x)},
$$
and 
\begin{align*}
\frac{\pa^2V_{\la}(x)}{\pa x_i\pa x_j} &= \la \frac{\pa V_{\la}(x)}{\pa x_j}\frac{(Qx)_i}{U(x)} + \la V_{\la}(x)\frac{\pa^2U(x)}{\pa x_i\pa x_j}  = \la^2V_{\la}(x)\frac{(Qx)_i}{U(x)}\frac{(Qx)_j}{U(x)} + \la V_{\la}(x)\frac{\pa^2U(x)}{\pa x_i\pa x_j} \\ & = \la^2V_{\la}(x)\frac{(Qxx'Q)_{ij}}{x'Qx} + \la V_{\la}(x)\frac{\pa^2U(x)}{\pa x_i\pa x_j}.
\end{align*}
Since $\langle Z_i, Z_j\rangle_t = a_{ij}\md t$ for $i, j = 1, \ldots, d$, and $Z(t) = W(t) + RL(t)$, we have:
\begin{align*}
\md V_{\la}(Z(t)) & = \frac12\SL_{i=1}^d\SL_{j=1}^d\frac{\pa^2V_{\la}(Z(t))}{\pa x_i\pa x_j}\md\langle Z_i, Z_j\rangle_t  + \SL_{i=1}^d\mu_i\frac{\pa V_{\la}(Z(t))}{\pa x_i}\md Z_i(t) \\ & = \frac12\SL_{i=1}^d\SL_{j=1}^da_{ij}\left(\la^2V_{\la}(Z(t))\frac{(QZ(t)Z(t)'Q)_{ij}}{Z(t)'QZ(t)} + \la V_{\la}(Z(t))\frac{\pa^2U(Z(t))}{\pa x_i\pa x_j}\right)\md t \\ & + \la\SL_{i=1}^d\mu_i\frac{(QZ(t))_i}{U(Z(t))}V_{\la}(Z(t))\md t + \SL_{j=1}^m\la\frac{QZ(t)}{U(Z(t))}V_{\la}(Z(t))r_{j}\md L_j(t) \\ & + \la V_{\la}(Z(t))\SL_{i=1}^d\mu_i\frac{(QZ(t))_i}{U(Z(t))}\md\left(W_i(t) - \mu_it\right)  \\ & = V_{\la}(Z(t))\be_{\la}(Z(t))\md t  + \md M(t) + \md \ol{L}(t),
\end{align*}
where for $x \in D\setminus\{0\}$ we let $\be_{\la}(x) := a(x)\la^2 - b(x)$, where
$$
a(x) = \SL_{i=1}^d\SL_{j=1}^da_{ij}\frac{(Qxx'Q)_{ij}}{x'Qx} = \frac{\tr(AQxx'Q)}{x'Qx} = \frac{\tr(x'QAQx)}{x'Qx} = \frac{x'QAQx}{x'Qx},
$$
and
$$
-b(x) := \theta(x) + \frac{x'Q\mu}{U(x)},\ \ \mathrm{where}\ \ 
\theta(x) := \frac12\SL_{i=1}^d\SL_{j=1}^da_{ij}\frac{\pa^2U(x)}{\pa x_i\pa x_j},
$$
and, in addition, 
$$
M(t) := \la\int_0^tV_{\la}(Z(t))\SL_{i=1}^d\mu_i\frac{(QZ(t))_i}{U(Z(t))}\md\left(W_i(t) - \mu_it\right),
$$
$$
\ol{L}(t) := \la\int_0^t\SL_{j=1}^m\frac{QZ(t)}{U(Z(t))}V_{\la}(Z(t))\cdot r_j\md L_j(t).
$$

\begin{lemma}
\label{lemma:L} 
The process $\ol{L} = (\ol{L}(t), t \ge 0)$ is nonincreasing a.s.
\end{lemma}

\begin{lemma} For $\la < \La$, there exist $r(\la), k(\la) > 0$ such that for $x \in D$, $\norm{x} \ge r(\la)$, we have: $\be_{\la}(x) < - k(\la)$. 
\label{lemma:bounds-for-drift}
\end{lemma}

Assuming we proved these two lemmata, let us complete the proof of Theorem~\ref{thm:main}. Fix $\la \in (0, \La)$. Take a compact set $C = \{x \in D\setminus \norm{x} \le r(\la)\}$, with $r(\la)$ from Lemma~\ref{lemma:bounds-for-drift}. By Proposition~\ref{prop:cpt-petite}, this set is petite. 
The process
$$
V_{\la}(Z(t)) - V_{\la}(Z(0)) - \int_0^tV_{\la}(Z(s))\be_{\la}(Z(s))\md s,\ \ t \ge 0,
$$
is a local supermartingale, because $M = (M(t), t \ge 0)$ is a local martingale and by Lemma~\ref{lemma:L}. Now, 
$$
\be_{\la}(x)V_{\la}(x) \le -k(\la)V_{\la}(x)1_{D\setminus C}(x) + b_{\la}1_C(x),
$$
for $b_{\la} := \max_{x \in C}\left[\be_{\la}(x)V_{\la}(x)\right]$. This maximum is well defined, because $\be_{\la}V_{\la}$ is a continuous function, and $C$ is a compact set.
The rest of the proof is trivial. 

\medskip

{\it Proof of Lemma~\ref{lemma:L}.} We can write $\ol{L}(t)$ as 
$$
\md\ol{L}(t) = \la\SL_{i=1}^m\frac{(QZ(t)R)_j}{U(Z(t))}V_{\la}(Z(t))\md L_j(t).
$$
But each $L_j$ can grow only when $Z \in D_j$, and then $(R'QZ(t))_j = QZ(t)\cdot r_j \le 0$. It suffices to note that $V_{\la}(Z(t)) \ge 0$ and $U(Z(t)) \ge 0$.

\medskip 

{\it Proof of Lemma~\ref{lemma:bounds-for-drift}.} 
For each $x \in D\setminus\{0\}$ we have: if $b(x) > 0$, then
$$
\la < \La(x) := \frac{b(x)}{a(x)}\ \ \Ra\ \ \be_{\la}(x) < 0.
$$
Note that $\theta(x) \to 0$ as $\norm{x} \to \infty$. From this and conditions (i), (ii) and (iii) of Theorem~\ref{thm:main} it is straightforward to see that
$$
\varliminf\limits_{\substack{\norm{x} \to \infty\\x \in D}}\La(x) = \La.
$$
Also, there exist $r_0, c_0 > 0$ such that for $x \in D$, $\norm{x} \ge r_0$ we have: $a(x), b(x) \ge c_0$. Now, fix $\la \in (0, \La)$. Then there exists $\de > 0$ such that $\de \le \la \le \La - 2\de$, and there exists $r(\la)$ such that for $x \in D,\ \norm{x} \ge r(\la)$ we have: $\La(x) \ge \La - \de$. Without loss of generality, we assume $r(\la) \ge r_0$. Now, for such $x$ we have:
$$
-\be_{\la}(x) \!=\! -a(x)\la^2 \!+\! b(x)\la = a(x)\la(- \la\! +\! \La(x)) \ge c_0\de((\La \!-\! \de) \!-\! (\La \!-\! 2\de)) \ge c_0\de^2.
$$
This completes the proof of Lemma~\ref{lemma:bounds-for-drift}. 

\section{Systems of Competing Brownian Particles}

\subsection{Classical Systems: Definitions and Background}

In this subsection, we use definitions from \cite{BFK2005}. Assume the usual setting: a filtered probability space $(\Oa, \CF, (\CF_t)_{t \ge 0}, \MP)$ with the filtration satisfying the usual conditions.  Let $N \ge 2$ (the number of particles). Fix parameters $g_1, \ldots, g_N \in \BR$ and $\si_1, \ldots, \si_N > 0$. We wish to define a system of $N$ Brownian particles in which the $k$th smallest particle moves a Brownian motion with drift $g_k$ and diffusion $\si_k^2$. 
Applications include: (i) mathematical finance, namely modeling the real-world feature of stocks with smaller capitalizations having larger growth rates and larger volatilities; it suffcies to take decreasing sequences $(g_k)$ and $(\si_k^2)$; (ii) diffusion limits of a certain type of exclusion processes, namely asymmetrically colliding random walks, see \cite{KPS2012}.

\begin{defn} Take i.i.d. standard $(\CF_t)_{t \ge 0}$-Brownian motions $W_1, \ldots, W_N$. For a continuous $\BR^N$-valued process 
$X = (X(t),\ t \ge 0)$, $X(t) = (X_1(t), \ldots, X_N(t))'$, let us define $\mP_t,\ t \ge 0$, the {\it ranking permutation} for the vector $X(t)$: this is a permutation on $\{1, \ldots, N\}$, such that:

(i) $X_{\mP_t(i)}(t) \le X_{\mP_t(j)}(t)$ for $1 \le i < j \le N$; 

(ii) if $1 \le i < j \le N$ and $X_{\mP_t(i)}(t) = X_{\mP_t(j)}(t)$, then $\mP_t(i) < \mP_t(j)$. 

Suppose the process $X$ satisfies the following SDE:
\begin{equation}
\label{mainSDE}
dX_i(t) = \SL_{k=1}^N1(\mP_t(k) = i)\left[g_k\, \md t + \si_k\, \md W_i(t)\right],\quad i = 1, \ldots, N.
\end{equation}
Then this process $X$ is called a {\it classical system of $N$ competing Brownian particles} with {\it drift coefficients} $g_1, \ldots, g_N$ and {\it diffusion coefficients} $\si_1^2, \ldots, \si_N^2$. For $i = 1, \ldots, N$, the component $X_i = (X_i(t), t \ge 0)$ is called the {\it $i$th named particle}. For $k = 1, \ldots, N$, the process
$$
Y_k = (Y_k(t),\ t \ge 0),\ \ Y_k(t) := X_{\mP_t(k)}(t) \equiv X_{(k)}(t),
$$
is called the {\it $k$th ranked particle}. They satisfy
$Y_1(t) \le Y_2(t) \le \ldots \le Y_N(t)$, $t \ge 0$. If $\mP_t(k) = i$, then we say that the particle $X_i(t) = Y_k(t)$ at time $t$ has {\it name} $i$ and {\it rank} $k$. 
\label{classical}
\end{defn}

Weak existence and uniqueness in law was established in \cite{BFK2005}. Consider the {\it gap process}: an $\BR^{N-1}_+$-valued process defined by
$$
Z = (Z(t), t \ge 0),\ \ Z(t) = (Z_1(t), \ldots, Z_{N-1}(t))', \ \ Z_k(t) = Y_{k+1}(t) - Y_k(t).
$$
It was shown in \cite{BFK2005} that this is an $\SRBM^{N-1}(R, \mu, A)$ in the orthant $S = \BR_+^{N-1}$ with parameters 
\begin{equation}
\label{R12}
R = 
\begin{bmatrix}
1 & -1/2 & 0 & 0 & \ldots & 0 & 0\\
-1/2 & 1 & -1/2 & 0 & \ldots & 0 & 0\\
0 & -1/2 & 1 & 0 & \ldots & 0 & 0\\
\vdots & \vdots & \vdots & \vdots & \ddots & \ddots & \ddots\\
0 & 0 & 0 & 0 & \ldots & 1 & -1/2\\
0 & 0 & 0 & 0 & \ldots & -1/2 & 1
\end{bmatrix},
\end{equation}
\begin{equation}
\label{mu}
\mu = \left(g_2 - g_1, g_3 - g_4, \ldots, g_N - g_{N-1}\right)',
\end{equation}
\begin{equation}
\label{A}
A = 
\begin{bmatrix}
\si_1^2 + \si_2^2 & -\si_2^2 & 0 & 0 & \ldots & 0 & 0\\
-\si_2^2 & \si_2^2 + \si_3^2 & -\si_3^2 & 0 & \ldots & 0 & 0\\
0 & -\si_3^2 & \si_3^2 + \si_4^2 & -\si_4^2 & \ldots & 0 & 0\\
\vdots & \vdots & \vdots & \vdots & \ddots & \vdots & \vdots\\
0 & 0 & 0 & 0 & \ldots & \si_{N-2}^2 + \si_{N-1}^2 & -\si_{N-1}^2\\
0 & 0 & 0 & 0 & \ldots & -\si_{N-1}^2 & \si_{N-1}^2 + \si_N^2
\end{bmatrix}.\nonumber
\end{equation}

\subsection{Main Results} In this subsection, we present results about the gap process: existence of a stationary distribution, Lyapunov functions and tail estimates. Let $\ol{g}_k := \left(g_1 + \ldots + g_k\right)/k$ for $k = 1, \ldots, N$. 

\begin{prop} The gap process has a stationary distribution if and only if 
\begin{equation}
\label{eq:gap-recurrence}
\ol{g}_k > \ol{g}_N\ \ \mathrm{for } k = 1, \ldots, N-1.
\end{equation}
In this case, it is $V$-uniformly ergodic with a certain function $V : \BR^{N-1}_+ \to [1, \infty)$. 
\end{prop}

\begin{proof} This result was already proved in \cite{BFK2005, Ichiba11, MyOwn6}, but for the sake of completeness we present a sketch of proof. The matrix $R$ is a reflection nonsingular $\CM$-matrix, and 
\begin{equation}
\label{eq:R-mu}
-R^{-1}\mu = 2\left(g_1 - \ol{g}_N, g_1 + g_2 - 2\ol{g}_N, \ldots, g_1 + \ldots + g_{N-1} - (N-1)\ol{g}_N\right)'.
\end{equation}
Define the quantities
\begin{equation}
\label{eq:def-b}
b_i = g_1 + g_2 + \ldots + g_i - i\ol{g}_N,\quad i = 1, \ldots, N-1.
\end{equation}
Then we can rewrite~\eqref{eq:R-mu} as
$$
-R^{-1}\mu = 2b,\ \ b = (b_1, \ldots, b_{N-1})'.
$$
Therefore, the gap process has a stationary distribution if and only if each component of this vector is strictly positive, which is equivalent to the condition~\eqref{eq:gap-recurrence}. In this case, the fluid path condition holds by \cite{Chen1996}, and so by \cite{BL2007} the gap process is $V$-uniformly ergodic for a certain Lyapunov function $V : \BR^{N-1}_+ \to [1, \infty)$. 
\end{proof}

From Corollary~\ref{cor:nonsingular-M}, we get a concrete Lyapunov function $V$, namely:
$$
V_{\la}(x) = e^{\la\phi(U(x))},\ \ U(x) := \left[x'R^{-1}x\right]^{1/2},
$$
where $\phi$ is defined in~\eqref{eq:phi}. We use the fact that the matrix $R$ is symmetric, so in the notation of Corollary~\ref{cor:nonsingular-M} we have: $C = I_{N-1}$ and $\ol{R} = R$. Here, we must have $\la < \La$, where
$$
\La := 2\min\limits_{x \in \BR^d_+\setminus\{0\}}\frac{|R^{-1}\mu\cdot x|U(x)}{x'R^{-1}AR^{-1}x}.
$$ 
Let us try to estimate the tail of the stationary distribution $\pi$. 

\begin{thm} Using the definition of $b_1, \ldots, b_{N-1}$ from~\eqref{eq:def-b}, we have:
$$
\int_{\BR^{N-1}_+}e^{\rho\norm{x}}\pi(\md x) < \infty\ \ \mathrm{for}\ \ \rho \in (0, \rho_0),\ \ \rho_0 := \frac2{\pi^2}\frac{\min(b_1, \ldots, b_{N-1})}{\norm{A}}N^{-2}.
$$
\end{thm}

\begin{proof} From the results of Theorem~\ref{thm:main}, we have: 
$$
\int_{\BR^{N-1}_+}e^{\rho\norm{x}}\pi(\mathrm{x}) < \infty\ \ \mathrm{for}\ \ \rho \in (0, K\La),
$$
where $K$ is defined in~\eqref{eq:K} (in this notation, $Q = R^{-1}$). Now, let us estimate $K\La$ from below. Define $\Sigma := \{x \in \BR^{N-1}_+\mid x_1 + \ldots + x_{N-1} = 1\}$. 

\begin{lemma} (i) The norm of the matrix $R^{-1}$ is equal to 
\begin{equation}
\label{eq:max-eigen}
\norm{R^{-1}} = \la^{-1}_1 = \left(1 - \cos\frac{\pi}N\right)^{-1};
\end{equation}

(ii) $U(x) \ge 1$ for $x \in \Sigma$;

(iii) $\left|R^{-1}\mu\cdot x\right| \ge 2\min(b_1, \ldots, b_{N-1})$ for $x \in \Sigma$;

(iv) $x'R^{-1}AR^{-1}x \le \norm{R^{-1}}^2\norm{A}$ for $x \in \Sigma$.
\label{lemma:many}
\end{lemma}

Suppose we proved Lemma~\ref{lemma:many}. From part (ii) we get: $K \ge 1$. Using (ii)-(iv), we obtain:
$$
\La \ge 2\frac{2\min(b_1, \ldots, b_{N-1})}{\norm{R^{-1}}^2\norm{A}}.
$$
Finally, using (i), we get:
$$
\La \ge 4\left(1 - \cos\frac{\pi}N\right)^{2}\frac{\min(b_1, \ldots, b_{N-1})}{\norm{A}}.
$$
But 
$$
1 - \cos\frac{\pi}N \ge \frac12\left(\frac{\pi}N\right)^2 = \frac{\pi^2}2\frac1{N^2}.
$$
Therefore, 
$$
\La \ge \frac2{\pi^2}\frac{\min(b_1, \ldots, b_{N-1})}{\norm{A}}N^{-2}. 
$$
The rest of the proof is trivial.

\medskip

{\it Proof of Lemma~\ref{lemma:many}.} (i) The eigenvalues of $R$ are given by (see, e.g., \cite{K2008})
$$
\la_k = 1 - \cos\frac{k\pi}N,\quad  k = 1, \ldots, N - 1.
$$
The eigenvalues of $R^{-1}$ are $\la_k^{-1},\ k = 1, \ldots, N - 1$. The matrix $R^{-1}$ is symmetric, so its norm is equal to the absolute value of its maximal eigenvalue. Therefore, we get~\eqref{eq:max-eigen}. 

(ii) The matrix $R^{-1}$ is symmetric and positive definite. Solving the optimization problem $x'R^{-1}x \to \min,\ x\cdot\munit = 1$, we get: the minimum is $\munit'R\munit$, which is equal to the sum of all elements of $R$, which, in turn, equals $1$.

(iii) Follows from the fact that $R^{-1}\mu < 0$ and~\eqref{eq:R-mu}. 

(iv) Follows from the multiplicative property of the Euclidean norm, and from the fact that for $x \in \Sigma$ we have: $\norm{x}^2 = x_1^2 + \cdots + x_{N-1}^2 \le (x_1 + \cdots + x_{N-1}^2)^2 = 1$. 
\end{proof}

\subsection{Asymmetric Collisions} One can generalize the classical system of competing Brownian particles from Definition~\ref{classical} in many ways. Let us describe one of these generalizations. Consider a classical system of competing Brownian particles, as in Definition~\ref{classical}. For $k = 1, \ldots, N-1$, let $L_{(k, k+1)} = (L_{(k, k+1)}(t), t \ge 0)$ 
be the semimartingale local time process at zero of the process $Z_k = Y_{k+1} - Y_k$. We shall call this the {\it collision local time} of the particles $Y_k$ and $Y_{k+1}$. For notational convenience, let $L_{(0, 1)}(t) \equiv 0$ and $L_{(N, N+1)}(t) \equiv 0$. 
Let
$$
B_k(t) = \SL_{i=1}^N\int_0^t1(\mP_s(k) = i){\mathrm{d}}W_i(s),\ \ k = 1, \ldots, N,\ \ t \ge 0.
$$
It can be checked that $\langle B_k, B_l\rangle_t \equiv \de_{kl}t$, so $B_1, \ldots, B_N$ are i.i.d. standard Brownian motions. 
As shown in \cite{BFK2005, Ichiba11}, the ranked particles $Y_1, \ldots, Y_N$ have the following dynamics:
$$
Y_k(t) = Y_k(0) + g_kt + \si_kB_k(t) - \frac12 L_{(k, k+1)}(t) + \frac12 L_{(k-1, k)}(t),\ \ k = 1, \ldots, N.
$$
The collision local time $L_{(k, k+1)}$ has a physical meaning of the push exerted when the particles $Y_k$ and $Y_{k+1}$ collide, which is needed to keep the particle $Y_{k+1}$ above the particle $Y_k$. Note that the coefficients at the local time terms are $\pm 1/2$. This means that the collision local time $L_{(k, k+1)}$ is split evenly between the two colliding particles: the lower-ranked particle $Y_k$ receives one half of this local time, which pushes it down, and the higher-ranked particle $Y_{k+1}$ receives the other one half of this local time, which pushes it up. In the paper \cite{KPS2012}, they considered systems of Brownian particles when this collision local time is split unevenly: the part $q^+_{k+1}L_{(k, k+1)}(t)$ goes to the upper particle $Y_{k+1}$, and the part $q^-_kL_{(k, k+1)}(t)$ goes to the lower particle $Y_k$. Let us give a formal definition.

\begin{defn} Fix $N \ge 2$, the number of particles. Take drift and diffusion coefficients
$g_1, \ldots, g_N \in \BR$, $\si_1, \ldots, \si_N > 0$, and, in addition, take {\it parameters of collision}
$$
q^{\pm}_1,\ldots, q^{\pm}_N \in (0, 1),\ \ q^+_{k+1} + q^-_k = 1,\ \ k = 1, \ldots, N-1.
$$
Consider a continuous adapted $\BR^N$-valued process $Y = \left(Y(t) = (Y_1(t), \ldots, Y_N(t))', t \ge 0\right)$. Take other $N-1$ continuous adapted real-valued nondecreasing processes
$$
L_{(k, k+1)} = (L_{(k, k+1)}(t), t \ge 0),\ \ k = 1, \ldots, N-1,
$$
with $L_{(k, k+1)}(0) = 0$, which can increase only when $Y_{k+1} = Y_k$:
$$
\int_0^{\infty}1(Y_{k+1}(t) > Y_k(t))\md L_{(k, k+1)}(t) = 0,\ \ k = 1, \ldots, N-1.
$$
Let $L_{(0, 1)}(t) \equiv 0$ and $L_{(N, N+1)}(t) \equiv 0$. Assume that
\begin{equation}
\label{asymm}
Y_k(t) = Y_k(0) + g_kt + \si_kB_k(t) - q^-_kL_{(k, k+1)}(t) + q^+_kL_{(k-1, k)}(t),\quad k = 1, \ldots, N.
\end{equation}
Then the process $Y$ is called the {\it system of competing Brownian particles with asymmetric collisions}. The gap process is defined similarly to the case of a classical system. 
\label{defn:asymm}
\end{defn}

Strong existence and pathwise uniqueness for these systems are shown in \cite[Section 2.1]{KPS2012}. When $q^{\pm}_1 = q^{\pm}_2 = \ldots = 1/2$, we are back in the case of symmetric collisions. 

\begin{rmk}
For systems of competing Brownian particles with asymmetric collisions,
we defined only ranked particles $Y_1, \ldots, Y_N$. It is, however, possible to define named particles $X_1, \ldots, X_N$ for the case of asymmetric collisions. This is done in \cite[Section 2.4]{KPS2012}. The construction works up to the first moment of a triple collision. A necessary and sufficient condition for a.s. absence of triple collisions is given in \cite{MyOwn3}. We will not make use of this construction in our article, instead working with ranked particles. 
\end{rmk}

It was shown in \cite{KPS2012} that the gap process for systems with asymmetric collisions, much like for the classical case, is an SRBM. Namely, it is an $\SRBM^{N-1}(R, \mu, A)$, where $\mu$ and $A$ are given by~\eqref{mu} and~\eqref{A}, and the reflection matrix $R$ is given by 
\begin{equation}
\label{R}
R = 
\begin{bmatrix}
1 & -q^-_2 & 0 & 0 & \ldots & 0 & 0\\
-q^+_2 & 1 & -q^-_3 & 0 & \ldots & 0 & 0\\
0 & -q^+_3 & 1 & -q^-_4 & \ldots & 0 & 0\\
\vdots & \vdots & \vdots & \vdots & \ddots & \vdots & \vdots\\
0 & 0 & 0 & 0 & \ldots & 1 & -q^-_{N-1}\\
0 & 0 & 0 & 0 & \ldots & -q^+_{N-1} & 1
\end{bmatrix}
\end{equation}
This matrix is also a reflection nonsingular $\CM$-matrix. Therefore, there exists a stationary distribution for this SRBM if and only if $R^{-1}\mu < 0$. In this case, we can apply the results of \cite{BL2007} again and conclude that the gap process is $V$-uniformly ergodic with a certain Lyapunov function $V : \BR^{N-1}_+ \to [1, \infty)$. 
Corollary~\ref{cor:nonsingular-M} allows us to find an explicit Lyapunov function and provide explicit tail estimates. A remark is in order: the matrix $R$ in~\eqref{R} in general is not symmetric, as opposed to the matrix $R$ in~\eqref{R12}. But for the following $(N-1)\times(N-1)$ diagonal matrix $C$, the matrix $\ol{R} = RC$ is diagonal:
$$
C = \diag\left(1, \frac{q^+_2}{q^-_2}, \frac{q^+_2q^+_3}{q^-_2q^-_3}, \ldots, \frac{q^+_2q^+_3\ldots q^+_{N-1}}{q^-_2q^-_3\ldots q^-_{N-1}}\right).
$$


\section*{Acknoweldgements}

The author would like to thank \textsc{Ioannis Karatzas} and \textsc{Ruth Williams} for help and useful discussion. Also, the author would like to thank two anonymous referees for pointing out misprints and useful comments which helped improve the paper. This research was partially supported by NSF Grants DMS 1007563, DMS 1308340, DMS 1409434, and DMS 1405210.

\end{document}